\documentclass[11pt,a4paper]{article}

\usepackage{caption2}
\usepackage{CJK}
\usepackage{tabls}
\usepackage{bm}
\usepackage{multirow}
\usepackage{amssymb}
\usepackage{amsfonts}
\usepackage{mathrsfs}
\usepackage{empheq}
\usepackage{amsmath}
\usepackage{amsthm}
\usepackage{graphicx}
\usepackage{color}
\usepackage{indentfirst}
\usepackage{hyperref}
\usepackage{float}
\usepackage{cases}
\usepackage[titletoc]{appendix}
\usepackage[colorinlistoftodos,english]{todonotes}

\allowdisplaybreaks%

\def\e{\eqref}

\def\i1n{i=1,\cdots,n}
\def\j1n{j=1,\cdots,n}
\def\ij1n{i,j=1,\cdots,n}

\def\R{\mathbb R}
\def\N{\mathbb N}

\def \i{\mathrm i}

 \numberwithin{equation}{section}
\theoremstyle{definition}

\def\R{{\bf R}}

\def\N{{\bf N}}

\def\d{\displaystyle}
\def\e{{\varepsilon}}

\newtheorem{thm}{Theorem}[section]

\newtheorem{lem}{Lemma}[section]
\newtheorem{prop}{Proposition}[section]
\newtheorem{rem}{Remark}[section]

\begin{document}

\begin{CJK*}{GB}{gbsn}
\title{Lifespan estimate for one dimensional wave equation with semilinear terms of spatial derivative }

\author{Ning-An Lai \thanks{Department of Mathematics, Shaoxing University, Shaoxing 312000, China (ninganlai@usx.edu.cn)}
\and Cui Ren \thanks{School of Mathematical Sciences, Zhejiang Normal University, Jinhua 321004, China (rencui@zjnu.edu.cn)}
\and Takiko Sasaki \thanks{Department of Mathematical Engineering, Faculty of Engineering, Musashino University, 3-3-3
Ariake, Koto-ku, Tokyo 135-8181, Japan (t-sasaki@musashino-u.ac.jp)}
\and Hiroyuki Takamura \thanks{Mathematical Institute, Tohoku University, Aoba, Sendai 980-8578, Japan (hiroyuki.takamura.a1@tohoku.ac.jp)} }

\maketitle

\begin{abstract}
This paper studies the upper and lower bounds of the lifespan for the classical solutions
to the initial value problems of one dimensional wave equations
with non-autonomous semilinear terms including the space-derivative of the unknown function.
This is a non-trivial business comparing to the analogous results with time-derivative type semilinear terms,
especially for the proof to obtain the sharp upper bound of the lifespan
as we have to deal with space dependent weights
among iteration procedures of  the weighted functional of the solution.
Also it is surprising that a part of them reaches to the same ordinary differential inequality for classical semilinear damped wave equations introduced by Li and Zhou (Discrete Contin. Dynam. Systems, 1995, 1(4): 503-520),
and we show a simple proof for blow up result from this ordinary differential inequality by iteration argument and slicing method in more general situation.
\end{abstract}




\section{Introduction}
In this paper, we are concerned with lifespan estimates of classical solutions
to the Cauchy problem of wave equations,
\begin{equation}
\label{IVP10}
\left\{
\begin{aligned}
& u_{tt} - u_{xx} = \frac{|u_x|^p}{\langle x\rangle^a} \quad\text{in $\R\times(0,T)$}, \\
& u(x,0)=\e f(x), \quad u_t(x,0)=\e g(x), \quad x\in\R,
\end{aligned}
\right.
\end{equation}
where $p>1,a\in\R$, $\e>0$ is a \lq\lq small" parameter and $ \langle x\rangle= (1+x^2)^{1/2} $ for any $x\in \R $.
The initial data $f$ and $g$ are smooth functions of compact support.
We are interested in the estimate of the lifespan $T(\e)$, i.e., the maximal existence time of the solution
defined by
\[
T(\e):=\sup\left\{T\in(0,\infty]\ :\ \mbox{the classical solution of (\ref{IVP10}) exists in the time interval $[0,T)$}\right\}.
\]
For autonomous case ($a=0$), it has been well-studied by
Sasaki, Takamatsu and Takamura \cite{STT2023} and Haruyama, Sasaki and Takamura \cite{HST}
that the lifespan estimates satisfy
\begin{equation}
\label{lifespan_auto}
T(\e)\sim C\e^{-(p-1)},
\end{equation}
where $T(\e)\sim A(\e,C)$ stands for the fact that
there exist two positive constants $C_1$ and $C_2$ independent of $\e$
such that
\[
A(\e,C_1)\le T(\e) \le A(\e,C_2).
\]
We note that the left inequality means that the solution exists at least to the time $A(\e,C_1)$ for any initial data,
namely it means long time existence of the solution,
while the right inequality means that there exist some special data for which the solution cannot exist
after the time $A(\e,C_2)$,
namely it implies the finite time blow-up of the solution.
\par
The early work on the semilinear wave equation with nonlinear term $|u_t|^p$, especially the upper bound of the lifespan,
is due to Zhou \cite{Zhou2001},
which also has the same upper bound of (\ref{lifespan_auto}).
This result is quite meaningful as it guarantees the sharpness of the general theory for nonlinear wave equations,
we refer to Takamura \cite{Takamura} for details and all the references on this direction.
The lower bound of lifespan was studied by Kitamura, Morisawa and Takamura \cite{KMT2023}
including non-autonomous nonlinear terms with space-variable dependent weight,
namely the problem for which $|u_x|^p$ is replaced with $|u_t|^p$ in (\ref{IVP10}) is studied.
The main difficulty to handle the nonlinear term of $|u_x|^p$ appears in the blow-up part
as we have to employ the special weighted functional which was introduced
by Rammaha \cite{Rammaha95, Rammaha97} for higher dimensional case like \cite{STT2023},
while we have a point-wise positivity of the solution for the nonlinear term $|u_t|^p$.
See \cite{HST} for details.
\par
The main purpose of this paper is to show that the lifespan $T(\e)$ of the Cauchy problem \eqref{IVP10} satisfies
\begin{equation}
\label{lifespan_non-auto}
\begin{array}{l}
T(\e)\sim
\left\{
\begin{array}{ll}
C\e^{-(p-1)/(1-a)} & \mbox{if}\ a<1,\\
\exp\left(C\e^{-(p-1)}\right) & \mbox{if}\ a=1,
\end{array}
\right.
\\
T(\e)=\infty \quad\mbox{if}\ a>1.
\end{array}
\end{equation}
This estimate is the same as the one in \cite{KMT2023}.
We note that the proof of the upper bound of the lifespan, blow-up,
in some case (Theorem \ref{thm2} below)
is reduced to a blow-up result on second order ordinary differential inequality (\ref{Gt}) below.
The remarkable fact is that it is exactly same as (3.18) in Theorem 3.1 by Li and Zhou \cite{LZ1995}
which establishes the sharp upper bound of the lifespan of solutions of the Cauchy problem for
classical semilinear damped wave equation in general space dimensions,
\[
u_{tt}-\Delta u+u_t=|u|^p\quad\mbox{in}\ \R^n\times(0,T).
\]
Their proof is organized the comparison and rescaling arguments,
and is available for $0\le a\le1$ in (\ref{Gt}) below.
But our proof employs only simple iteration argument and slicing method,
and succeeds to extend the result for any $a\le1$.
See Lemma \ref{2lem1} below as well as Remark \ref{rem:Li-Zhou}.
We note that the iteration argument is originally introduced to nonlinear hyperbolic equations
by John \cite{John1979} to show the sub-critical blow-up for semilinear wave equations in three space dimensions.
And the slicing method is firstly introduced by Agemi, Kurokawa and Takamura \cite{AKT2000}
to show the critical blow-up for systems semilinear wave equations in three space dimensions.
It is still useful to shorten the proof of the blow-up result especially in the critical case.
See Shao, Sasaki and Takamura \cite{SST2025} for example.

\par
This paper is organized as follows.
(\ref{lifespan_non-auto}) is divided into three theorems according to three cases in Section 2.
The blow-up part of (\ref{lifespan_non-auto}) is proved in Section 3 and 4 with different conditions
on the data.
The existence part of (\ref{lifespan_non-auto}) is shortly proved in Section 5.


\section{Main results}
Throughout  of this paper, we assume that $(f,g)\in C_0^2(\R)\times C_0^1(\R)$ satisfy
\begin{align}
\label{supp_data}
\mathrm{supp}\, f, \mathrm{supp}\, g \subset \{ x \in \mathbb{R} : |x| \leq R \}, \quad R \geq 1.
\end{align}
Let $u$ be a classical solution of (\ref{IVP10}).
Then it is well-known that the support condition of the initial data implies that
\begin{align}
\label{supp_sol}
\mathrm{supp}\, u(x,t)  \subset \{ (x,t) \in \R\times [0,T] : |x| \leq t+ R \}.
\end{align}

\begin{thm}
\label{thm1}
Assume \eqref{supp_data} and,
\begin{equation}
\label{fg1}
f(x), g(x) \geq 0 \ \text{and} \ f(x) \not\equiv 0,
\end{equation}
or
\begin{equation}
\label{fg2}f(x), g(x)\geq0  \ \mbox{and} \ g(x)\not\equiv 0.
\end{equation}
Then there exists a positive constant $\e_1=\e_1(f,g,p,a,R)$ such that the classical solution of the Cauchy problem \eqref{IVP10} cannot exist as far as $T$ satisfies
\begin{equation}\label{T1}
T>
\left\{
\begin{array}{ll}
C\e^{-(p-1)/(1-a)} & \mbox{if}\ a<1,\\
\exp\left(C\e^{-(p-1)}\right) & \mbox{if}\ a=1
\end{array}
\right.
\end{equation}
for all $\e\in(0,\e_1]$, where $C$ is a positive constant independent of $ \varepsilon$.
\end{thm}

\begin{rem}
The conclusion of the above theorem implies that
\[
T(\e)\le\left\{
\begin{array}{ll}
C\e^{-(p-1)/(1-a)} & \mbox{if}\ a<1,\\
\exp\left(C\e^{-(p-1)}\right) & \mbox{if}\ a=1.
\end{array}
\right.
\]
\end{rem}

\begin{thm}
\label{thm2}
Assume \eqref{supp_data} and
\begin{equation}
\label{initial}
\begin{aligned}
\int_{\R}(-e^x+e^{-x})f'(x)dx\ge0,\
\int_{\R}(-e^x+e^{-x})\{g'(x)-f'(x)\}dx\ge0(\not\equiv0).
\end{aligned}
\end{equation}
Then, the same conclusion as in Theorem \ref{thm1} holds.
\end{thm}

\begin{rem}
Note that integration by parts yields
\[
\int_{\R}(-e^x+e^{-x})f'(x)dx=\int_{\R}(e^x+e^{-x})f(x)dx,
\]
so that the condition (\ref{initial}) can be written by $f$ and $g$ instead of $f'$ and $g'$.
\end{rem}

\begin{thm}
\label{thm3}
Assume \eqref{supp_data}.
Then there exists a positive constant $\e_2=\e_2(f,g,p,a,R)$ such that
the classical solution of the Cauchy problem \eqref{IVP10} exists as far as $T$ satisfies
\[
T<
\left\{
\begin{array}{ll}
C\e^{-(p-1)/(1-a)} & \mbox{if}\ a<1,\\
\exp\left(C\e^{-(p-1)}\right) & \mbox{if}\ a=1,\\
\infty & \mbox{if}\ a>1
\end{array}
\right.
\]
for all $\e\in(0,\e_2]$,
where $C$ is a positive constant independent of $ \varepsilon$.
\end{thm}

\begin{rem}
The conclusion of the above theorem implies that
\[
\begin{array}{l}
T(\e)\ge
\left\{
\begin{array}{ll}
C\e^{-(p-1)/(1-a)} & \mbox{if}\ a<1,\\
\exp\left(C\e^{-(p-1)}\right) & \mbox{if}\ a=1,
\end{array}
\right.
\\
T(\e)=\infty \quad\mbox{if}\ a>1.
\end{array}
\]
\end{rem}

All the theorems above are proved with the concrete expression of a solution $u$ of (\ref{IVP10}),
\begin{equation}
\label{integral}
u(x,t)=\e u^0(x,t)+L_{a}(|u_x|^p)(x,t),
\end{equation}
where $u^0$ is a solution of the free wave equation with the same initial data;
\begin{equation}
\label{linear}
u^0(x,t):=\frac{1}{2}\{f(x+t)+f(x-t)\}+\frac{1}{2}\int_{x-t}^{x+t}g(y)dy,
\end{equation}
and a linear integral operator $L_{a}$ for a function $v=v(x,t)$ is Duhamel's term defined by
\begin{equation}
\label{nonlinear}
L_a(v)(x,t):=\frac12\int_0^t ds \int_{x-t+s}^{x+t-s}
\frac{v(y,s)}{(1+y^2)^{a/2}}\,dy.
\end{equation}


\section{Proof of Theorem \ref{thm1}}
In this section, we give the proof of Theorem \ref{thm1}, the case 1 and case 2 of which will be presented in subsection \ref{sec11} and \ref{sec12}, respectively.

\subsection{Proof of Theorem \ref{thm1} under the assumption (\ref{fg1})}
\label{sec11}

\par\noindent
\textbf{The subcritical case}.
\par
Before proceeding with the proof for the subcritical case, we first establish a key lemma.
\begin{lem}
\label{lem1}
Let $a\in\R$ and $p>1$.
Assume that $ a<1 $, and the function $H=H(t)\in C([E,T))$ satisfies
\begin{align}
\label{lem11}& H(t)\geq D_1t^{2-a/p},\\
\label{lem12}& H(t) \geq D_2 \int_E^tds\int_E^sr^{1-2p-a/p}|H(r)|^pdr
\end{align}
for $t\ge E$ with some positive constants $D_1$, $D_2$ and $E$.
Then, there exists a positive constant $D=D(D_2,p,a)$ such that $H(t) $ cannot exist if
\begin{equation}
\label{estimate:T}
T>(DD_1^{-1})^{\frac{p-1}{1-a}}.
\end{equation}
where $$ D= \frac{ p^{\frac{2p}{(p-1)^2}} 2^{\frac{2p+1-a}{p-1}} (2p+1-a )^{\frac{2}{p-1}} }{ \left(D_2(p-1)^2\right)^{\frac{1}{p-1} }}, $$
which satisfies
\[
 (DD_1^{-1})^{\frac{p-1}{1-a}} > 2E.
 \]
\end{lem}

\par\noindent
\textbf{Proof.}
From \eqref{lem11}, it is easy to get
$$ H(t)\geq D_1 (t-E)^{2+(1-a)/p} t^{-1/p}.  $$
Under the assumption \eqref{lem12}, we may claim that
\begin{align}\label{ht}
 H(t) \geq C_n (t-E)^{a_n} t^{-b_n - \frac{1}{p}}, \ t \ge E,
\end{align}
with the recurrence relations
\begin{equation}
\label{seq}
\left\{
\begin{aligned}
& a_{n+1}= pa_n +2+\frac{1-a}{p}, \ \ a_1=(1-a)/p +2, \\
& b_{n+1}= pb_n +2p, \ \  b_1=0,\\
& C_{n+1}=\frac{D_2 C_n^p}{(a_{n+1})^2 } , \ \ C_1=D_1,
\end{aligned}
\right.
\end{equation}
where $ a_n,b_n \geq0. $

In the following we are going to prove the above claim. It holds by \eqref{lem12} that
\begin{align}\label{diedai}
 H(t) \geq D_2 t^{-2p-1/p}  \int_{E}^tds\int_{E}^sr (r-E)^{(1-a)/p} |H(r)|^pdr.
\end{align}
Plugging \eqref{ht} into \eqref{diedai} yields
\begin{align*}
H(t)& \geq D_2 C_n^p t^{-2p-1/p}  \int_{ E}^tds\int_{E}^sr (r-E)^{(1-a)/p}(r-E)^{a_n p} r^{-b_n p-1}  dr\\
& \geq D_2 C_n^p t^{-2p-1/p-b_n p } \int_{E}^tds\int_{E}^s (r-E)^{(1-a)/p +a_n p} dr\\
& \geq \frac{D_2 C_n^p}{( a_n p+2+(1-a)/p)^2} t^{-2p-1/p-b_n p } (t-E)^{(1-a)/p +a_n p+2},
\end{align*}
which implies the claim \eqref{ht} with \eqref{seq}.

Solving the recurrence relations \eqref{seq}, we obtain
\begin{align}
\label{an}& a_n= \frac{2p+1-a}{p-1}p^{n-1} - \frac{2p+1-a}{p(p-1)},\\
\label{bn}& b_n =\frac{2p}{p-1} p^{n-1}- \frac{2p}{p-1}.
\end{align}
Obviously we have by \eqref{an}
\begin{align*}
& a_{n+1} \leq \frac{2p+1-a}{p-1}p^{n},
\end{align*}
which yields by $\eqref{seq}_3$
\begin{align}\label{cn1}
C_{n+1}& \geq p^{-2n} \frac{D_2C_n^p(p-1)^2}{(2p+1-a)^2}= D_3 \frac{C_{n}^p}{p^{2n}}
\end{align}
with
\[
D_3 = \frac{D_2(p-1)^2}{(2p+1-a)^2}.
\]
It follows by \eqref{cn1}
\begin{align*}
\log C_{n+1} & \geq p \log C_n  - (2n)\log p + \log D_3,
\end{align*}
which further leads to
\begin{align*}
& \log C_{n+1} - (n+1)\frac{2 \log p}{p - 1} - \left( \frac{2 \log p}{(p-1)^2} - \frac{\log D_3}{p-1} \right)\\
& \geq p \left( \log C_n - \frac{2 \log p}{p - 1} n - \left( \frac{2 \log p}{(p-1)^2} - \frac{\log D_3}{p-1}\right)\right).
\end{align*}
Iterating the above inequality to $n=1$, we come to
\begin{align}\label{cn}
\log C_n \geq p^{n-1} \log \left( \frac{D_1 D_3^{\frac{1}{p-1}}}{ p^{\frac{2p}{(p-1)^2}} } \right) + n\frac{2 \log p}{p-1} + \log \left( \frac{ p^{\frac{2}{(p-1)^2}} }{ D_3^{\frac{1}{p-1}} } \right).
\end{align}
Substituting \eqref{an}, \eqref{bn} and \eqref{cn} into \eqref{ht} yields
\begin{align}\label{jt}
H(t)& \geq C_n (t-E)^{a_n} t^{-b_n - \frac{1}{p}}\nonumber\\
&\geq \left( \frac{D_1 D_3^{\frac{1}{p-1}}}{ p^{\frac{2p}{(p-1)^2}} } \right)^{p^{n-1}} p^{\frac{2n(p-1)+2}{(p-1)^2}} D_3^{-\frac{1}{p-1}}(t-E)^{a_n} t^{-b_n - \frac{1}{p}}.
\end{align}
Hence for $ t\geq 2E $, it follows that
\begin{align}\label{3-1}
H(t)&\geq \left( \frac{D_1 D_3^{\frac{1}{p-1}}}{ p^{\frac{2p}{(p-1)^2}} } \right)^{p^{n-1}} p^{\frac{2n(p-1)+2}{(p-1)^2}} D_3^{-\frac{1}{p-1}}\left( \frac{t}{2} \right)^{a_n} t^{-b_n - \frac{1}{p}}\nonumber\\
& \geq \left( \frac{D_1 D_3^{\frac{1}{p-1}} t^{ \frac{1-a}{p-1}} }{ p^{\frac{2p}{(p-1)^2}} 2^{\frac{2p+1-a}{p-1}} } \right)^{p^{n-1}} p^{\frac{2n(p-1)+2}{(p-1)^2}} D_3^{-\frac{1}{p-1}} 2^{-\frac{2p+1-a}{p(1-p)}} t^{ \frac{2p^2-3p+a}{p(p-1)}}.
\end{align}
Hence if
\begin{align}\label{T}
 \frac{D_1 D_3^{\frac{1}{p-1}} t^{ \frac{1-a}{p-1}} }{ p^{\frac{2p}{(p-1)^2}} 2^{\frac{2p+1-a}{p-1}} } >1,
 \end{align}
then we find that $ H(t)\rightarrow \infty $ by taking the limit as $n\rightarrow \infty $ in \eqref{3-1}. This means the upper bound of lifespan $T$ for $H(t)$ satisfies
$$ T \leq  \left(D_1^{-1} D\right)^{\frac{p-1 }{1-a}}$$
with
$$ D= \frac{ p^{\frac{2p}{(p-1)^2}} 2^{\frac{2p+1-a}{p-1}} (2p+1-a )^{\frac{2}{p-1}} }{ (D_2(p-1)^2)^{\frac{1}{p-1} }}. $$
This completes the proof of Lemma \ref{lem1}.
\hfill$\square$

Going back to our problem \eqref{IVP10}, we define the functional
\begin{align*}
H(t):=\int_{1}^t (t-s) ds\int_{s+R_0}^{s+R} \frac{u(x,s)}{\langle x\rangle^{a/p}}dx,
\end{align*}
where $$ 1 \leq R_0<R. $$ It is easy to see
\begin{align}\label{c5c6}
D_5 \int_{1}^t \frac{(t-s)}{s^{a/p}}ds\int_{s+R_0}^{s+R} u(x,s)dx \leq H(t) \leq D_6 \int_{1}^t \frac{(t-s)}{s^{a/p}}ds\int_{s+R_0}^{s+R} u(x,s)dx
\end{align}
with
\begin{align*}
& D_5:=\min \left\{ \frac{1}{\big( 2(R+1)\big)^{a/p}}, \,1  \right\},\ D_6:=\max \left\{ \frac{1}{\big( 2(R+1)\big)^{a/p}}, \,1  \right\}.
\end{align*}

Differentiating to $H(t)$ with respect to time twice yields
\begin{align}\label{hde10}
H''(t)&= \int_{t+R_0}^{t+R}\frac{u(x, t)}{\langle x\rangle^{a/p}}dx\nonumber\\
& \geq \frac{D_5}{t^{a/p}} \int_{t+R_0}^{t+R} u(x,t)dx \nonumber\\
& \geq \frac{D_5 \varepsilon}{2}\frac{1}{t^{a/p}} \int_{t+R_0}^{t+R} \left\{ f(x+t) + f(x-t) +  \int_{x-t}^{x+t} g(y)dy \right\} dx\nonumber\\
& \ \ \ \ + \frac{D_5}{2}\frac{1}{t^{a/p}} F(t),
\end{align}
where
\begin{align}\label{Ft}
F(t) := \int_{t+R_0}^{t+R} dx \int_0^t ds \int_{x-t+s}^{x+t-s} \frac{|u_x(y,s)|^p}{ \langle y\rangle^a   } dy.
\end{align}
From \eqref{supp_data} and \eqref{hde10}, we obtain
\begin{align}\label{lin}
H^{\prime\prime}(t) & \geq\frac{D_5\varepsilon}{2 t^{a/p}} \int_{t+R_{0}}^{t+R} f(x-t) dx \geq  \frac{ C_f D_5\varepsilon}{2 t^{a/p}},
\end{align}
where
\begin{align*}
& C_f := \int_{R_{0}}^{R} f(y) dy>0,
\end{align*}
by assuming that there exists $  x_0 \in (R_0, R) $ such that $f(x_0)>0 $.
Integrating \eqref{lin} over $[1,t]$, noting that \( a < 1 \) and hence \( 1 -   a/p > 0 \),
we obtain
\begin{align*}
H'(t) & \geq \frac{C_f D_5 \varepsilon}{2} \int_{1}^t s^{ -  a/p} ds + H'(1) \\
& = \frac{C_f D_5 \varepsilon}{2(1 - a/p) } \left[ s^{1 -  a/p} \right]_{1}^{t}
\end{align*}
for $ t \geq 1 .$ Integrating the above inequality again yields
\begin{align*}
H(t) & \geq \frac{ C_f D_5\varepsilon}{2(1 -  a/p)} \int_{1}^t (s^{1 - a/p}-1) \, ds +H(1) \nonumber\\
& \geq \frac{ C_f D_5 (1-(1/2)^{1-a/p})\varepsilon}{2(1 -  a/p)}  \int_{2}^t s^{1 - a/p} \, ds
\end{align*}
for $ t \geq 2.$  It can be calculated that
\begin{align}\label{cf}
H(t) & \geq \frac{ C_f D_5 (1-(1/2)^{1-a/p})(1-(1/2)^{2-a/p})\varepsilon}{2(1 -  a/p) (2 - a/p)}  t^{2 - a/p}\nonumber\\
& =  D_7 \varepsilon t^{2 - a/p}
\end{align}
for $ t \geq 4,$ where
\begin{align}
& D_7:=\frac{ C_f D_5 (1-(1/2)^{1-a/p})(1-(1/2)^{2-a/p})}{2(1 -  a/p) (2 - a/p)}.
\end{align}

It follows from the definition \eqref{Ft} that
\begin{align}\label{11-1}
F(t)=\int_{0}^{t}ds\int_{t+R_0}^{t+R}dx\int_{x-t+s}^{x+t-s}\frac{|u_x(y,s)|^p}{ \langle y\rangle^a  }dy.
\end{align}
In the following we assume that
\begin{align*}
t\geq R_{1}:=\frac{R-R_0}{2}>0,
\end{align*}
then similar to \cite{STT2023}, it holds for $0\leq s\leq t-R_1$
\begin{align}\label{11-2}
& \int_{t+R_{0}}^{t+R}dx\int_{x-t+s}^{x+t-s}\frac{|u_x(y,s)|^p}{\langle y\rangle^a}dy \nonumber\\
& =  \left( \int_{s+R_0}^{s+R} \int_{t+R_0}^{y+t-s} + \int_{s+R}^{2t-s+R_0} \int_{t+R_0}^{t+R} + \int_{2t-s+R_0}^{2t-s+R} \int_{y-t+s}^{t+R} \right) \frac{|u_x(y,s)|^p}{\langle y\rangle^a} dxdy \nonumber\\
& \geq\int_{s+R_0}^{s+R}dy\int_{t+R_0}^{y+t-s}\frac{|u_x(y,s)|^p}{ \langle y\rangle^a}dx,
\end{align}
while for \( t - R_1\leq s \leq t \), we have
\begin{align}\label{11-3}
&\int_{t+R_0}^{t+R} dx \int_{x-t+s}^{x+t-s} \frac{|u_x(y,s)|^p}{  \langle y\rangle^a} dy \nonumber\\
& = \left( \int_{s+R_0}^{2t-s+R_0} \int_{t+R_0}^{y+t-s} + \int_{2t-s+R_0}^{s+R} \int_{y-t+s}^{y+t-s} + \int_{s+R}^{2t-s+R} \int_{y-t+s}^{t+R} \right) \frac{|u_x(y,s)|^p}{\langle y\rangle^a} dxdy \nonumber\\
& \geq \int_{s+R_0}^{2t-s+R_0} dy \int_{t+R_0}^{y+t-s} \frac{|u_x(y,s)|^p}{ \langle y\rangle^a} dx + \int_{2t-s+R_0}^{s+R} dy \int_{y-t+s}^{y+t-s} \frac{|u_x(y,s)|^p}{\langle y\rangle^a} dx.
\end{align}
Combining \eqref{11-1}-\eqref{11-3}, we arrive at
\begin{align}\label{ft}
F(t) & \geq \int_{0}^{t-R_1} \frac{t - s}{t} ds \int_{s+R_0}^{s+R} (y - s - R_0) \frac{|u_x(y,s)|^p}{\langle y\rangle^a} dy \nonumber\\
& \quad + \int_{t-R_1}^{t} \frac{t - s}{t} ds \int_{s+R_0}^{2t-s+R_0} (y - s - R_0) \frac{|u_x(y,s)|^p}{ \langle y\rangle^a} dy \nonumber\\
& \quad + \int_{t-R_1}^{t} 2(t-s) ds \int_{2t-s+R_0}^{s+R} \frac{y - s - R_0}{2t} \frac{|u_x(y,s)|^p}{\langle y\rangle^a} dy \nonumber\\
& \geq \frac{1}{t} \int_{0}^{t} (t - s) ds \int_{s+R_0}^{s+R} (y - s - R_0) \frac{|u_x(y,s)|^p}{\langle y\rangle^a} dy \nonumber\\
& \geq \frac{D_8}{t} \int_{0}^{t} \frac{ (t - s)}{ s^{a}} ds \int_{s+R_0}^{s+R} (y - s - R_0) |u_x(y,s)|^pdy,
\end{align}
where
\[
D_8:=\min \left\{ \frac{1}{\left( 2(R+1)\right)^{a}}, \,1  \right\}.
\]
By combining \eqref{hde10} and \eqref{ft}, we obtain
\begin{align}\label{101}
H''(t)& \geq \frac{D_5}{2} t^{-a/p} F(t) \nonumber\\
& \geq
 \frac{D_5D_8}{2}  t^{-1-a/p} \int_{0}^{t} (t - s) ds \int_{s+R_0}^{s+R} \frac{(y - s - R_0)|u_x(y,s)|^p }{ s^{a}} dy .
\end{align}
On the other hand, employing H$\ddot{\mbox{o}}$lder's inequality and integration by parts yields
\begin{align}\label{102}
|H(t)| & \leq D_6 \left| \int_{0}^{t}  (t - s) ds \int_{s+R_0}^{s+R} \frac{\partial_y (y - s - R_0) u(y, s)}{s^{a/p}} dy \right|  \nonumber\\
& \leq D_6 \int_{0}^{t}(t - s)ds \int_{s+R_0}^{s+R} \frac{ (y - s - R_0) |u_x(y, s)| } {s^{a/p}} dy \nonumber\\
& \leq D_6 \left( \int_{0}^{t} (t - s) ds \int_{s+R_0}^{s+R} \frac{(y - s - R_0)|u_x(y,s)|^p }{ s^{a}} dy \right)^{1/p} I(t)^{1-1/p},
\end{align}
where
\begin{align*}
I(t)& := \int_{0}^t (t - s) ds \int_{s + R_0}^{s + R} (y - s - R_0) dy\\
& = \frac{1}{4} t^2 (R-R_0)^2 = t^2 R_1^2.
\end{align*}
It then concludes from \eqref{101} and \eqref{102}
\begin{align*}
H''(t)& \geq \frac{D_5D_8}{2} D_6^{-p} R_{1}^{-2(p-1)} t^{-1-2(p-1)} t^{-a/p} |H(t)|^p \\
& \geq D_9 t^{1-2p-a/p} |H(t)|^p
\end{align*}
with
\[
D_9:=\frac{D_5D_8}{2} D_6^{-p} R_{1}^{-2(p-1)}
\]
and $ t \geq R_1 $.
Integrating the above inequality with respect to $t$ twice yields
\begin{align}
\label{frame10}
H(t) \geq D_9 \int_{R_1}^tds\int_{R_1}^sr^{1-2p-a/p}|H(r)|^pdr
\quad \text{for } t \geq R_1.
\end{align}
Employing Lemma \ref{lem1} to \eqref{cf} and \eqref{frame10}, the solution of the Cauchy problem \eqref{IVP10} cannot exist if $t$ satisfies
$$ t > C_1 \e^{-(p-1)/(1-a)} $$
with
$$ C_1= \left(  \frac{ p^{\frac{2p}{(p-1)^2}} 2^{\frac{2p+1-a}{p-1}} (2p+1-a )^{\frac{2}{p-1}} }{ (D_9(p-1)^2)^{\frac{1}{p-1} } D_7} \right)^{\frac{p-1}{1-a}}.$$
Hence we obtain the lifespan estimate $\eqref{T1}_1$ for the subcritical case
for $ \varepsilon\in(0, \varepsilon_1]  $, where
\begin{align}\label{e11}
\varepsilon_1=( 2R_1)^{-\frac{1-a}{p-1}} (C_1)^{\frac{1-a}{p-1}},
\end{align}
in this case, which is determined by $ C_1 \e^{-(p-1)/(1-a)}  > 2R_1 $.
\hfill$\square$

\par\noindent
\textbf{The critical case.}

Next, we focus the lifespan estimate $\eqref{T1}_2$ for the critical case.
Following Shao, Takamura and Wang \cite{STW2025}, we define
\begin{align*}
& H(t):=\int_{0}^t(t-s)ds\int_{s+R_0}^{s+R}x^{-1}u(x,s)dx
\end{align*}
for $ t\geq 0$. Then it is easy to get
\begin{align}\label{cht}
H''(t) & =\int_{t+R_0}^{t+R}x^{-1}u(x,t)dx\nonumber\\
&\geq \frac{1}{t+R} \int_{t+R_0}^{t+R} u(x,t)dx \nonumber\\
&  \geq \frac{\varepsilon}{2(t+R)} \int_{t+R_0}^{t+R} \left\{ f(x+t) + f(x-t) + \int_{x-t}^{x+t} g(y)dy \right\} dx\nonumber\\
& \ \ \ \ + \frac{1}{2(t+R)} F(t),
\end{align}
where
\begin{align*}
F(t) & = \int_{t+R_0}^{t+R} dx \int_0^t ds \int_{x-t+s}^{x+t-s} \frac{|u_x(y,s)|^p}{ \langle y\rangle   } dy, \ \ \text{for} \ a=1.
\end{align*}
Noting the assumption for the initial data in Theorem \ref{thm1}, we obtain
\begin{align}\label{H1}
H^{\prime\prime}(t)\geq\frac{\varepsilon}{2(t+R)} \int_{t+R_0}^{t+R}f(x-t)dx
\geq \frac{C_f\varepsilon}{2(t+R)},
\end{align}
where
\begin{align*}
& C_f=  \int_{R_0}^{R}f(y)dy.
\end{align*}
Noting that $ H^{\prime}(0)=0 $, it then holds from \eqref{H1}
\begin{align*}
H'(t) & \ge \frac{C_f\varepsilon}{2} \int_{0}^{t} \frac{1} {s+R} ds + H^{\prime}(0)\\
& =  \frac{C_f\varepsilon}{2}  \log \frac{t+R}{R},
\end{align*}
which in turn results in
\begin{align}\label{ccf}
H(t) & \ge \frac{C_f\varepsilon}{2}  \int_{0}^{t}\log \frac{s+R}{R} ds+ H(0) \nonumber\\
&  \ge \frac{C_f\varepsilon}{2}  \int_{\frac{t}{2}}^{t}\log \frac{s+R}{R} ds  \nonumber\\
& \ge \frac{C_f\varepsilon}{4} t \log \frac{t+2R}{2R}  \nonumber\\
& \ge \frac{C_f\varepsilon}{4} t \log \frac{t}{2R}
\end{align}
for $t\geq 2R$.

On the other hand, similar to \eqref{ft} we have
\begin{align}\label{ch}
H''(t)& \geq \frac{1}{2(t+R)} F(t) \nonumber\\
& \geq \frac{1}{2t(t+R)} \int_{0}^{t} (t - s) ds \int_{s+R_0}^{s+R} (y - s - R_0) \frac{|u_x(y,s)|^p}{ \langle y\rangle} dy.
\end{align}
It is easy to see
\begin{align}\label{sy}
& \frac{1}{  \langle y\rangle } \geq  \frac{R_0}{2R}\frac{1}{(s+R_0)}=:\frac{D_{10}}{(s+R_0)}.
\end{align}
Combining \eqref{ch} and \eqref{sy} yields
\begin{align}\label{cch}
&H''(t) \geq \frac{D_{10} }{4t^2} \int_{0}^{t} \frac{(t - s)} { s+R_0} ds \int_{s+R_0}^{s+R}  (y - s - R_0) |u_x(y,s)|^p dy
\end{align}
for $ a=1, t\geq R $. Utilizing H$\ddot{\mbox{o}}$lder's inequality, we have
\begin{align}\label{cm}
|H(t)| & =  \left |  \int_{0}^t(t-s)ds\int_{s+R_0}^{s+R}y^{-1}u(y,s)dy \right| \nonumber\\
& \leq \left| \int_{0}^t \frac{t-s}{s+R_0} ds\int_{s+R_0}^{s+R} u(y,s)dy \right| \nonumber\\
& \leq \int_{0}^t \frac{t-s}{s+R_0} ds\int_{s+R_0}^{s+R} (y - s - R_0)|u_x(y,s)|dy \nonumber\\
& \leq   \left( \int_{ 0}^{t}\frac{t-s}{s+R_0} ds \int_{s+R_0}^{s+R}(y - s - R_0)|u_x(y,s)|^p  dy \right)^{1/p} M(t)^{1-1/p},
\end{align}
where
\begin{align*}
M(t)& := \int_{ 0}^{t}\frac{t-s}{s+R_0} ds \int_{s+R_0}^{s+R}(y - s - R_0) dy\\
& = \frac{1}{2} (R-R_0)^2  \int_{ 0}^{t}\frac{t-s}{s+R_0} ds\\
& = \frac{1}{2} (R-R_0)^2 \int_{R_0}^{t+R_0}\frac{t+R_0-m}{m} dm \\
& \leq \frac{1}{2} (R-R_0)^2 \left(t \log \frac{t+R_0}{R_0} + R_0 \log \frac{t+R_0}{R_0}\right)\\
& \leq  4R_1^2 t \log \frac{t+R_0}{R_0}
\end{align*}
for $ t\geq R_0 $. If $R_0>1$, the inequality $\log\frac{t+R_0}{R_0}\leq \log t$ holds provided that $t\geq \frac{R_0}{R_0-1}$. When $R_0=1$, a direct computation yields $\log\frac{t+R_0}{R_0} = \log(t+1) \leq 2\log t$ for all $t\geq 2$. Consequently, we arrive at
\[
M(t) \leq 8R_1^2\, t \log t, \qquad t \geq 2R.
\]

Combining \eqref{cch} and \eqref{cm}, we obtain
\begin{align}\label{H2}
H''(t) & \geq
D_{10} 4^{-1} 8^{-(p-1)} R_1^{-2(p-1)} t^{-2-(p-1)}  \log^{-(p-1)}t |H(t)|^p \nonumber\\
& \geq  D_{11} t^{-(p+1)} \log ^{-(p-1)}t |H(t)|^p,
\end{align}
with
\begin{align*}
& D_{11}:=D_{10} 4^{-1} 8^{-(p-1)}R_1^{-2(p-1)} .
\end{align*}
Integrating \eqref{H2} with respect to time twice yields
\begin{align}\label{frame_critical}
H(t) & \geq D_{11} \int_{2R}^tds\int_{2R}^sr^{-(p+1)}\log^{-(p-1)}rH(r)^pdr
\end{align}
for  $t\ge 2R$.

The following lemma comes from Theorem 1 in \cite{SST2025}, which will be used to get the desired lifespan estimate.
\begin{lem}\label{lem3}
Let \( H \in C([R, T]) \) be a solution to
\begin{equation}
\label{seq*}
\left\{
\begin{aligned}
& H(t) \geq A t^a (\log t)^{-b} \left( \log \frac{t}{R} \right)^c, & t \in [R,T), \\
& H(t) \geq B (\log t)^x \int_R^t ds \int_R^s r^y \left( \log \frac{r}{R} \right)^z |H(r)|^p dr, & t \in [R,T),
\end{aligned}
\right.
\end{equation}
where all $A>0, B>0 $ and $ T>R>1 $ are constants. Assume that exponents $ a,b,c,x,y,z,p $ satisfy
\begin{equation}
    \begin{cases}
        \,p > 1, a \leq 1, \, b \geq \max\left\{0, \d\frac{x}{p-1}\right\}, \\
        \,y + pa = -1, z + cp > -1, \, z + cp \geq c-1.
    \end{cases}
\end{equation}
\( T \) has to satisfies
    \begin{equation}
        T \leq \exp\left(\max\left\{2\log R_{\infty}, (A^{-1}D)^{\frac{p-1}{x+z+1+(c-b)(p-1)}}\right\}\right),
    \end{equation}
    where
    \begin{align*}
        &R_{\infty}= R \prod_{k=1}^{\infty}(1 + 2^{-k}), \\
        &D= 2^{\left(c+\frac{p+(z+1)(p-1)}{(p-1)^2}\right)} p^{\frac{p}{(p-1)^2}} \left(\frac{1}{B}\max\left\{c+\frac{z+1}{p-1}, c+\frac{z+1}{p}\right\}\right)^{\frac{1}{p-1}} > 0.
    \end{align*}
\end{lem}
Employing Lemma \ref{lem3} to \eqref{ccf} and \eqref{frame_critical}, and setting
\[
a=1,\ b=0,\ c= 1,\ x=-(p-1),\ y=-(p+1),\ z= 0
\]
which satisfy
\[
\begin{cases}
p > 1, \ a \leq 1, \ b \geq \max\left\{0, -1\right\}, \\
y + pa = -1, \ z + cp > -1, \ z + cp \geq 0,
\end{cases}
\]
 we derive the upper bound of lifespan to the solution of \eqref{IVP10} for the critical case $a=1$
\begin{align*}
T(\varepsilon)\leq\exp\left(\max\left\{2\log R_\infty,C_2\e^{-(p-1)}\right\} \right),
\end{align*}
where
\[
C_2= \left[  C_f^{-1} 2 ^{\frac{2 p-1}{(p-1)^2}+3} p^{\frac{p}{(p-1)^2}} D_{11}^{-1} \left(\frac{p}{p-1}\right)^{\frac{1}{p-1}} \right]^{p-1}\ \mbox{and}\ R_{\infty} = 2R \prod_{k=1}^{\infty} (1 + 2^{-k}).
\]
Hence, we obtain the lifespan estimate $\eqref{T1}_2$ for $a=1$
\begin{align*}
T(\varepsilon)\leq\exp\left(C_2 \e^{-(p-1)} \right),
\end{align*}
for $\e\in(0,\e_1]$, where
\begin{align}\label{e12}
\varepsilon_1= ( 2\log R_\infty)^{-\frac{1}{p-1}}  ( C_2 )^{\frac{1}{p-1}}
\end{align}
in this case.

\subsection{Proof of Theorem \ref{thm1} under the assumption (\ref{fg2})}
\label{sec12}
In this subsection, we are going to prove that Theorem \ref{thm1} also holds for
\begin{align}\label{g0}
&f(x), g(x)\geq0  \ \mbox{and} \ g(x)\not\equiv 0.
\end{align}

\par\noindent
\textbf{The subcritical case}.

From \eqref{hde10}, we have
\begin{align}\label{sg1}
H''(t) &\geq \frac{D_5 \varepsilon}{2t^{ a/p}} \int_{t + R_0}^{t + R} dx \int_{x - t}^{x + t} g(y) \, dy \nonumber\\
&\geq \frac{D_5 \varepsilon}{2t^{ a/p}} \int_{t + R_0}^{t + (R+R_0)/2} dx \int_{x - t}^{x + t} g(y) \, dy \nonumber\\
&\geq \frac{D_5 \varepsilon}{2t^{ a/p}} \int_{t + R_0}^{t + (R+R_0)/2} dx \int_{(R + R_0)/2}^{R + R_0} g(y) \, dy,
\end{align}
where we assume that
\[
2t + R_0 > \frac{R + R_0}{2} \Longleftrightarrow t \geq \frac{R_1}{2}, \quad R_1 := \frac{R - R_0}{2}.
\]
It further yields by \eqref{sg1}
\begin{align}\label{sg2}
H''(t) \geq \frac{C_g D_5 \varepsilon}{2t^{ a/p}} \int_{t + R_0}^{t + (R+R_0)/2} dx \quad \text{for } t \geq R_1,
\end{align}
where
\[
C_g := \int_{(R + R_0)/2}^R g(y) \, dy > 0.
\]
From the initial condition \eqref{g0}, we have $u > 0 $ for $ t \geq 0 $, combined with \eqref{hde10} and \eqref{Ft} yields $H''(t) > 0$ for $t\ge 1$.
Noting that $H'(1)=0$ and setting $E_1=\max\{1,R_1\}$, we see that $H'(E_1)>0$ if $R_1>1$ and $H'(E_1)=0$ if $0 <R_1\le1$. This implies  $H'(E_1)\ge0$ for $ t \geq E_1$. Hence integrating \eqref{sg2} yields
\begin{align*}
	H'(t) & \geq \frac{C_g D_5 R_1 \varepsilon}{2} \int_{E_1}^t s^{ -a/p} \, ds + H'(E_1) \\
	& \geq \frac{C_g D_5 R_1 \varepsilon}{2(1 -  a/p)} \left[ s^{1 -  a/p} \right]_{E_1}^{t} \quad \text{for } t \geq E_1.
\end{align*}
Also noting that \( 1 -  a/p > 0 \), which leads to
\[
H(t) \geq \frac{C_g D_5 R_1 \varepsilon}{2(1 - a/p)} \int_{E_1}^t \left[ s^{1 -a/p} - E_1^{1 - a/p} \right] ds + H(E_1) \quad \text{for } t \geq E_1,
\]
which further yields
\[
H(t) \geq \frac{(1 - (1/2)^{1 -  a/p}) C_g D_5 R_1 \varepsilon}{2(1 -  a/p)} \int_{2E_1}^t s^{1 - a/p} \, ds \quad \text{for } t \geq 2E_1,
\]
and hence
\begin{align}\label{cg}
H(t) & \geq \frac{(1 - (1/2)^{1 -  a/p})(1 - (1/2)^{2 - a/p}) C_g D_5 R_1 \varepsilon}{2(1 -  a/p) (2 -  a/p)} t^{2 -  a/p}\nonumber\\
& \geq D_{12}\varepsilon t^{2 - a/p},\quad \text{for } t \geq 4E_1
\end{align}
with
$$ D_{12}:= \frac{(1 - (1/2)^{1 -  a/p})(1 - (1/2)^{2 - a/p}) C_g D_5 R_1 }{2(1 -  a/p) (2 -  a/p)}.$$

From \eqref{frame10}, we have
\begin{align}
\label{g1}
H(t) & \geq D_9 \int_{R_1}^tds\int_{R_1}^sr^{1-2p-a/p}|H(r)|^pdr \nonumber\\
&  \geq D_9 \int_{4E_1}^tds\int_{4E_1}^sr^{1-2p-a/p}|H(r)|^pdr
\quad \text{for } t \geq 4E_1.
\end{align}
Employing  Lemma \ref{lem1} with $ E=4E_1 $ to \eqref{cg} and \eqref{g1}, we then conclude that the solution of the Cauchy problem \eqref{IVP10} cannot exist if $T$ satisfies
$$ T > C_1 \e^{-(p-1)/(1-a)} $$
with
$$ C_1= \left(  \frac{ p^{\frac{2p}{(p-1)^2}} 2^{\frac{2p+1-a}{p-1}} (2p+1-a )^{\frac{2}{p-1}} }{ (D_9(p-1)^2)^{\frac{1}{p-1} } D_{12}} \right)^{\frac{p-1}{1-a}}.$$
Moreover $ \varepsilon$ satisfies $ \varepsilon\in(0, \varepsilon_1]  $, where
\begin{align}\label{e11}
\varepsilon_1=( 8E_1)^{-\frac{1-a}{p-1}} (C_1)^{\frac{1-a}{p-1}}.
\end{align}
This condition guarantees that  $ C_1 \e^{-(p-1)/(1-a)}  > 8E_1 $.

\par\noindent
\textbf{The critical case}.

By \eqref{cht} it holds
\begin{align*}
H''(t) &\geq \frac{ \varepsilon}{2(t+R)} \int_{t + R_0}^{t + R} dx \int_{x - t}^{x + t} g(y) \, dy.
\end{align*}
Similar to \eqref{sg2}, we have
\begin{align*}
H''(t) &\geq  \frac{ \varepsilon}{2(t+R)} \int_{t + R_0}^{t + (R+R_0)/2} dx \int_{(R + R_0)/2}^{R + R_0} g(y) \, dy\\
& \geq \frac{C_g \varepsilon}{2(t+R)} \int_{t + R_0}^{t +(R+R_0)/2} dx,
\end{align*}
where
\[
C_g := \int_{(R + R_0)/2}^R g(y) \, dy > 0.
\]
Integrating the above inequality over $[R_1,t]$ yields
\[
H'(t) \geq \frac{C_g  R_1 \varepsilon}{2} \int_{R_1}^t \frac{1}{s+R} \, ds + H'(R_1) \quad \text{for } t \geq R_1,
\]
which further yields
\[
H'(t) \geq \frac{C_g  R_1 \varepsilon}{2} \log \frac{t+R}{R_1+R} \geq  \frac{C_g  R_1 \varepsilon}{2} \log \frac{t+R}{2R}\quad \text{for } t \geq R_1,
\]
and
\begin{align}\label{ccg}
H(t) & \geq \frac{C_g  R_1 \varepsilon}{2} \int_{R_1}^t \log \frac{s+R}{2R}+ H(R_1)\nonumber\\
& \geq \frac{C_g  R_1 \varepsilon}{2} \int_{t/2}^{t} \log \frac{s+R}{2R} ds\nonumber\\
& \geq \frac{C_g  R_1 \varepsilon}{2} \frac{t}{2}  \log \frac{t+2R}{4R}\nonumber\\
& \geq \frac{C_g  R_1 \varepsilon}{4} t  \log \frac{t}{4R},~~~for~  t \geq 2R_1.
\end{align}

On the other hand, by \eqref{frame_critical}, we have
\begin{align}
\label{19-1}
H(t) & \geq D_{11} \int_{R}^tds\int_{R}^sr^{-(p+1)}\log^{-(p-1)}rH(r)^pdr \nonumber\\
& \geq D_{11} \int_{4R}^tds\int_{4R}^sr^{-(p+1)}\log^{-(p-1)}rH(r)^pdr,~~~for~  t\ge 4R.
\end{align}
Employing Theorem 1 in \cite{SST2025} to \eqref{ccg} and \eqref{19-1}, we derive the upper bound lifespan of the solution to the Cauchy problem \eqref{IVP10}
\begin{align*}
T(\varepsilon)\leq\exp\left(\max\left\{2\log R_\infty,C_2\e^{-(p-1)}\right\} \right),
\end{align*}
where $ C_2= \left[  C_g^{-1}R_1^{-1} 2 ^{\frac{2 p-1}{(p-1)^2}+3} p^{\frac{p}{(p-1)^2}} D_{11}^{-1} \left(\frac{p}{p-1}\right)^{\frac{1}{p-1}} \right]^{p-1}$ and $ R_{\infty} = 4R \prod_{k=1}^{\infty} (1 + 2^{-k}).$
For $\e\in(0,\e_1]$, we obtain that  $ T $ satisfies
\begin{align*}
T(\varepsilon)\leq\exp\left(C_2 \e^{-(p-1)} \right),
\end{align*}
where
\begin{align}\label{e12}
\varepsilon_1= ( 2\log R_\infty)^{-\frac{1}{p-1}}  ( C_2 )^{\frac{1}{p-1}}
\end{align}
in this case.


\section{Proof of Theorem \ref{thm2}}
In this section we give the proof to Theorem \ref{thm2},
which demonstrates an alternative proof for the blow-up result for the Cauchy problem \eqref{IVP10} under different assumption for the initial data. Before proceeding with the main proof, we first establish a key lemma.

\begin{lem}[extended Li-Zhou theorem]
\label{2lem1}
Let $a\in\R,a\leq1$ and $p>1$.
 Assume that a function $G=G(t)\in C([E,T))$ satisfies
\begin{align}
\label{2lem11}&G(t)\geq M_1,\\
\label{2lem12}& G(t)\geq M_2 \int_E^te^{-2s}ds\int_E^s \frac{e^{2r}|G(r)|^p} {(1+r)^{a}}dr
\end{align}
for some positive constants $E, M_1$, $M_2$.
Then there exists a positive constant $M=M(M_2,p,a)$ such that $G(t) $ cannot exists for
\[
T > \left\{
\begin{array}{ll}
 (M M_1^{-1})^{(p-1)/(1-a)}, & \mbox{if}\ a<1,\\
\exp \left(  (M_1^{-1} M)^{p-1}  \right), & \mbox{if}\ a=1,
\end{array}
\right.
\]
where
\[
M = \left\{
\begin{array}{ll}
  (2p)^{\frac{4p-2}{(p-1)^2}} 2^{\frac{(1-a)(p-1+p^{n_{0}})}{(p-1)^2}} (1-a )^{\frac{1}{p-1}}  (M_2(p-1))^{-\frac{1}{p-1}},  & \mbox{if}\ a\leq0,\\
 (2p)^{\frac{4p-2}{(p-1)^2}} 2^{\frac{(1+a)(p-1)+p^{n_{0}}}{(p-1)^2}} \left(   M_2(p-1) \right)^{-\frac{1}{p-1}}, & \mbox{if}\ 0<a<1,\\
 2^{ \frac{2p-1}{(p-1)^2}} p^{\frac{p}{(p-1)^2}} \left(\frac{4}{M_2(p-1)}\right)^{\frac{1}{p-1}}, & \mbox{if}\ a=1. \\
\end{array}
\right.
\]
\end{lem}

\begin{rem}
\label{rem:Li-Zhou}
(\ref{2lem12}) comes from (\ref{Gt}) below which is exactly same as (3.18)
in Theorem 3.1 by Li and Zhou \cite{LZ1995}.
But, as stated in Introduction, our result is new for $a<0$.
\end{rem}

\par\noindent
\textbf{Proof}.
\textbf{The subcritical case: $a<1$}.
We divide the proof into two cases:\\
Following Chen and Palmieri \cite{CP2020}, or its variants of their method in the subsequent papers \cite{CP2021,CR2021,P2021,PT2022,PT2023}, we employ the product-type slicer to apply the slicing method in the iteration argument to \eqref{2lem12} which has exponential multipliers.\\
\textbf{Case 1:  $ a\leq0$}. First we introduce
$$ l_{n} := \prod_{i=1}^{n} \left(1 + \frac{1}{(2p)^{i}}\right).  $$
Then for $E\ge 1$, we claim that
\begin{align}\label{23-1}
 G(t)\ge K_n(t-l_nE)^{m_n},~~~t\ge l_nE,
\end{align}
with the recurrence relations
\begin{equation}
\label{24-1}
\left\{
\begin{aligned}
& m_{n+1}= m_np+1-a,  \ \  m_1=0, \\
& K_{n+1}=\frac{M_2K_n^p}{(2p)^{n+2}(m_np-a+1)} \left( \frac{l_n}{l_{n+1}}  \right)^{m_np-a},\ \  K_1=M_1. \\
\end{aligned}
\right.
\end{equation}
Obviously if $\eqref{24-1}_1$ holds, then $ m_n \geq0 $ for all $ n\in \mathbb N $.

Substituting \eqref{23-1} into \eqref{2lem12} yields
\begin{align}\label{24-4}
G(t)& \geq M_2K_n^p \int_{l_{n+1}E}^te^{-2s}ds\int_{l_{n-1}s/l_{n}}^s \frac{e^{2r} (r-l_nE)^{m_np} } {(1+r)^{a}}dr \nonumber\\
& \geq M_2K_n^p \int_{l_{n+1}E}^te^{-2s}ds\int_{l_{n}s/l_{n+1}}^s \frac{e^{2r} (r-l_nE)^{m_np} } {(1+r)^{a}}dr \nonumber\\
& \geq M_2K_n^p \int_{l_{n+1}E}^te^{-2s}ds\int_{l_{n}s/l_{n+1}}^s e^{2r} (r-l_nE)^{m_np-a} dr \nonumber\\
& \geq M_2K_n^p \int_{l_{n+1}E}^t e^{-2s} \left(\frac{l_n}{l_{n+1}}s- l_nE\right)^{m_np-a} ds\int_{l_{n}s/l_{n+1}}^s e^{2r} dr
\end{align}
for $ t\ge l_{n+1}E$, which further yields
\begin{align}\label{26-1}
G(t)& \geq  \frac{M_2K_n^p}{2}  \int_{l_{n+1}E}^t\left(\frac{l_n}{l_{n+1}}s-l_nE\right)^{m_np-a} \left(1-e^{2(l_n/l_{n+1}-1)s}\right) ds.
\end{align}
Utilizing the inequality $ e^{-x}\leq \frac{1}{1+x} $ for $ x > 0 $ and the condition $ E\ge 1 $, we obtain
\begin{align}\label{26-2}
1-e^{2(l_n/l_{n+1}-1)s} &\ge1-\frac{1}{1+2(l_{n+1}-l_n)E} \nonumber\\
& =\frac{2(l_{n+1}-l_n)E}{1+2(l_{n+1}-l_n)E} \nonumber\\
&\geq \frac{2}{ (2p)^{n+1}+2 } \nonumber\\
& \geq \frac{2}{ (2p)^{n+2}}.
\end{align}
Combining \eqref{26-1} and \eqref{26-2}, we have
\begin{align*}
G(t)& \geq \frac{M_2K_n^p}{(2p)^{n+2}} \int_{l_{n+1}E}^t\left(\frac{l_n}{l_{n+1}}s-l_nE\right)^{m_np-a} ds\\
& \geq \frac{M_2K_n^p}{(2p)^{n+2}} \left( \frac{l_n}{l_{n+1}}  \right)^{m_np-a} \int_{l_{n+1}E}^t\left(s-l_{n+1}E\right)^{m_np-a} ds\\
& \geq \frac{M_2K_n^p}{(2p)^{n+2}(m_np-a+1)} \left( \frac{l_n}{l_{n+1}}  \right)^{m_np-a} \left(t-l_{n+1}E\right)^{m_np-a+1},
\end{align*}
which implies the claim \eqref{23-1} and \eqref{24-1}.
Iterating $\eqref{24-1}_1$ to $n=1$ yields
\begin{align*}
& m_{n+1}= p^{n} \frac{1-a}{p-1}-\frac{1-a}{p-1},
\end{align*}
which implies
\begin{align*}
& m_{n+1} \leq p^{n} \frac{1-a}{p-1}.
\end{align*}
Hence we have
\begin{align}\label{1122-1}
K_{n+1}&= \frac{M_2K_n^p}{(2p)^{n+2}m_{n+1}} \left( \frac{l_n}{l_{n+1}}  \right)^{m_{n+1}-1} \nonumber\\
& \geq  \frac{M_2K_n^p}{(2p)^{n+2}(p^{n} \frac{1-a}{p-1})} \left( \frac{l_n}{l_{n+1}}  \right)^{p^{n} \frac{1-a}{p-1}-1}.
\end{align}

The last factor in \eqref{1122-1} should be estimated delicately. Setting
\begin{align*}
b_{n}& = \left(\frac{l_{n}}{l_{n+1}}\right)^{p^{n}}  = \left(1 + \frac{1}{(2p)^{n+1}}\right)^{-p^{n}},
\end{align*}
then
\begin{align*}
	\log b_n &= -p^{n}\log\left(1+\frac{1}{(2p)^{n+1}}\right)\\
	& = -p^n\left[\frac{1}{(2p)^{n+1}} + o\left(\frac{1}{(2p)^{n+1}}\right)\right] \\
	&= -\frac{1}{2p}\left(\frac{1}{2}\right)^n + o\left(\left(\frac{1}{2}\right)^n\right).
\end{align*}
Since $\frac12<1$, it follows that
\[
\lim_{n\to\infty}\log b_n = \lim_{n\to\infty}\left[-\frac{1}{2p}\left(\frac12\right)^n + o\left(\left(\frac12\right)^n\right)\right] = 0,
\]
which yields
\[
\lim_{n\to\infty} b_n = 1.
\]
Hence there exists $n_0\in \mathbb{N}$ such that
\begin{equation}\label{bn1}
	\begin{aligned}
b_{n} = \left(\frac{l_{n}}{l_{n+1}}\right)^{p^{n}} \ge \left(\frac12\right)^{p^{n_0}},\ \ \text{for}\ n>n_0.
\end{aligned}
	\end{equation}
Also noting that
\[
\frac12\le \frac{l_{n}}{l_{n+1}}<1,
\]
then it holds that
\begin{equation}\label{bn2}
	\begin{aligned}
b_{n} = \left(\frac{l_{n}}{l_{n+1}}\right)^{p^{n}} \ge \left(\frac12\right)^{p^{n_0}},\ \ \text{for}\ n=1, 2, \cdots, n_0.
\end{aligned}
	\end{equation}
We then conclude from \eqref{bn1} and \eqref{bn2} that for all $n\in \mathbb{N}$
\begin{equation}\label{bn3}
	\begin{aligned}
b_{n} = \left(\frac{l_{n}}{l_{n+1}}\right)^{p^{n}} \ge \left(\frac12\right)^{p^{n_0}}.
\end{aligned}
	\end{equation}
Combining \eqref{1122-1} yields
\begin{align}\label{24-2}
K_{n+1} & \geq  \frac{M_2K_n^p}{(2p)^{n+2}(p^{n} \frac{1-a}{p-1})} \left( \frac{l_n}{l_{n+1}}  \right)^{p^{n} \frac{1-a}{p-1}-1}\nonumber\\
&= \frac{M_2K_n^p}{(2p)^{n+2}(p^{n} \frac{1-a}{p-1})} (b_n)^{\frac{1-a}{p-1}}\Bigg/\frac{l_{n}}{l_{n+1}}\nonumber\\
& \geq  \frac{M_2K_n^p}{2^{ \frac{p^{n_{0}}(1-a)}{p-1}  } (2p)^{n+2}(p^{n} \frac{1-a}{p-1})}  \nonumber\\
& \geq  \frac{M_3 K_n^p}{ (2p)^{2n}}
\end{align}
with
$$ M_3:=  \frac{M_2}{2^{ \frac{p^{n_{0}}(1-a)}{p-1}  }(2p)^{2}(\frac{1-a}{p-1})}. $$
This implies
\begin{align*}
\log K_{n+1} & \geq p \log K_n  - (2n)\log(2p) + \log M_3,
\end{align*}
which is equivalent to
\begin{align*}
& \log K_{n+1} - \frac{2 \log (2p)}{p - 1} (n+1) - \left( \frac{2 \log (2p)}{(p-1)^2} - \frac{\log M_3}{p-1} \right)\\
& \geq p \left( \log K_n - \frac{2 \log (2p)}{p - 1} n - \left( \frac{2 \log (2p)}{(p-1)^2} - \frac{\log M_3}{p-1}\right)\right).
\end{align*}
Therefore, we may iterate to $n=1$ to get
\begin{align}\label{24-3}
\log K_n \geq p^{n-1} \log \left( \frac{M_1 M_3^{\frac{1}{p-1}}}{ (2p)^{\frac{2p}{(p-1)^2}} } \right) + n \frac{2 \log (2p)}{p-1} + \log \left( \frac{ (2p)^{\frac{2}{(p-1)^2}} }{ M_3^{\frac{1}{p-1}} } \right).
\end{align}
We then get from \eqref{23-1} and \eqref{24-3}
\begin{align*}
& G(t) \geq  \left( \frac{M_1 M_3^{\frac{1}{p-1}}}{ (2p)^{\frac{2p}{(p-1)^2}} } \right)^{p^{n-1}} (2p)^{\frac{2n(p-1)+2}{(p-1)^2}} M_3^{-\frac{1}{p-1}}(t-l_nE)^{m_n}
\end{align*}
for $ t\ge l_nE $.

Setting
\[
 l_{\infty} := \prod_{i=1}^{\infty} \left(1 + \frac{1}{(2p)^{i}}\right),
\]
then we have for $ t\ge 2l_\infty E $
\begin{align*}
G(t)& \geq  \left( \frac{M_1 M_3^{\frac{1}{p-1}}}{ (2p)^{\frac{2p}{(p-1)^2}} } \right)^{p^{n-1}} (2p)^{\frac{2n(p-1)+2}{(p-1)^2}} M_3^{-\frac{1}{p-1}}(t-l_\infty E)^{m_n} \\
& \geq  \left( \frac{M_1 M_3^{\frac{1}{p-1}}}{ (2p)^{\frac{2p}{(p-1)^2}} } \right)^{p^{n-1}} (2p)^{\frac{2n(p-1)+2}{(p-1)^2}} M_3^{-\frac{1}{p-1}}\left(\frac{1}{2} t \right)^{m_n}\\
& \geq  \left( \frac{M_1 M_3^{\frac{1}{p-1}} t^{\frac{1-a}{p-1}}}{ (2p)^{\frac{2p}{(p-1)^2}} 2^{\frac{1-a}{p-1}} } \right)^{p^{n-1}} (2p)^{\frac{2n(p-1)+2 }{(p-1)^2}} M_3^{-\frac{1}{p-1}}\left(\frac{1}{2} t \right)^{-\frac{1-a}{p-1}}.
\end{align*}
Hence if $ t $ satisfies
$$ \frac{M_1 M_3^{\frac{1}{p-1}} t^{\frac{1-a}{p-1}}}{ (2p)^{\frac{2p}{(p-1)^2}} 2^{\frac{1-a}{p-1}} } >1,  $$
then $ G(t) \rightarrow \infty$ as $ n\rightarrow \infty $, which means the upper bound of lifespan satisfies
$$ T \leq  (M_1^{-1} M)^{\frac{p-1 }{1-a}}$$
with
$$ M= \frac{ (2p)^{\frac{4p-2}{(p-1)^2}} 2^{\frac{(1-a)(p-1+p^{n_{0}})}{(p-1)^2}} (1-a )^{\frac{1}{p-1}} }{ (M_2(p-1))^{\frac{1}{p-1} }} . $$

We now consider the second subcritical case.
\\
\textbf{Case 2:  $0<a<1$}. 
In this case we are going to establish
\begin{align}\label{th21}
 G(t)\ge T_n \frac{(t-l_nE)^{h_n}}{{(1+t)}^{j_n}} \quad\mbox{for}\ t\ge l_nE,
\end{align}
with non-negative $ h_n, j_n, T_n $  satisfying
\begin{equation}
\label{th22}
\left\{
\begin{aligned}
& h_{n+1}= h_np+1,  \ \  h_1=0, \\
& j_{n+1}=j_np+a,  \ \  j_1=0, \\
& T_{n+1}=\frac{M_2T_n^p}{(2p)^{n+2}(h_np+1)} \left( \frac{l_n}{l_{n+1}}  \right)^{h_np},\ \  T_1=M_1. \\
\end{aligned}
\right.
\end{equation}
Similar to \eqref{24-4}, we have
\begin{align}\label{th23}
G(t) & \geq M_2T_n^p \int_{l_{n+1}E}^te^{-2s}ds\int_{l_{n}s/l_{n+1}}^s \frac{e^{2r} (r-l_nE)^{h_np} } {(1+r)^{j_np+a}}dr,
\end{align}
which yields for $0<a<1 $
\begin{align}
G(t) & \geq M_2T_n^p (1+t)^{-j_np-a} \int_{l_{n+1}E}^te^{-2s}ds\int_{l_{n}s/l_{n+1}}^s e^{2r} (r-l_nE)^{h_np} dr,
\end{align}
which further yields
\begin{align*}
 G(t) & \geq \frac{M_2T_n^p}{(2p)^{n+2} (h_n p+1) } \left(\frac{l_n}{l_{n+1}}\right)^{h_n p} (1+t)^{-j_np-a} (t-l_{n+1} E)^{h_n p+1}.
\end{align*}
Hence we prove the inequality \eqref{th21} with \eqref{th22}.

Iterating $\eqref{th22}_1$ and $\eqref{th22}_2$ to $n=1$ we get
\begin{align}\label{th25}
& h_{n}= \frac{p^{n-1}}{p-1}-\frac{1}{p-1},  \\
\label{jn} & j_{n}= p^{n-1}\frac{a}{p-1}-\frac{a}{p-1},
\end{align}
combining \eqref{bn3} yields
\begin{align}\label{24-2}
T_{n+1} & = \frac{M_2T_n^p}{(2p)^{n+2}(h_np+1)} \left( \frac{l_n}{l_{n+1}}  \right)^{h_np} \nonumber\\
& \geq \frac{M_2(p-1)T_n^p}{(2p)^{n+2}p^{n}} \left( \frac{l_n}{l_{n+1}}  \right)^{\frac{p^n}{p-1}}\Bigg/ \left(\frac{l_{n}}{l_{n+1}}\right)^{\frac{p}{p-1}} \nonumber\\
& \geq  \frac{M_2(p-1)}{2^{ \frac{p^{n_0}}{p-1}}(2p)^{2}(2p)^{2n} } T_n^p \nonumber\\
& =   \frac{M_4}{ (2p)^{2n}}T_n^p,
\end{align}
where
$$ M_4 := \frac{M_2(p-1)}{2^{ \frac{p^{n_0}}{p-1}} (2p)^{2}} .$$
Similar to \eqref{24-3}, we obtain from \eqref{24-2}
\begin{align}\label{th24}
\log T_n \geq p^{n-1} \log \left( \frac{M_1 M_4^{\frac{1}{p-1}}}{ (2p)^{\frac{2p}{(p-1)^2}} } \right) + n\frac{2 \log (2p)}{p-1} + \log \left( \frac{ (2p)^{\frac{2}{(p-1)^2}} }{ M_4^{\frac{1}{p-1}} } \right).
\end{align}
Combining \eqref{th21} and \eqref{th24}, we have
\begin{align*}
G(t) & \ge T_n \frac{(t-l_nE)^{h_n}}{{(1+t)}^{j_n}} \\
&  \geq  \left( \frac{M_1 M_4^{\frac{1}{p-1}}}{ (2p)^{\frac{2p}{(p-1)^2}} } \right)^{p^{n-1}} (2p)^{\frac{2n(p-1)+2}{(p-1)^2}} M_4^{-\frac{1}{p-1}}\frac{(t-l_nE)^{h_n}}{{(1+t)}^{j_n}}
\end{align*}
for $ t\ge l_nE $. We then combine \eqref{th25} and \eqref{jn} to conclude that for $ t\ge 2l_\infty E $ we have
\begin{align*}
G(t)&  \geq  \left( \frac{M_1 M_4^{\frac{1}{p-1}}}{ (2p)^{\frac{2p}{(p-1)^2}} } \right)^{p^{n-1}} (2p)^{\frac{2n(p-1)+2}{(p-1)^2}} M_4^{-\frac{1}{p-1}} \left(\frac{1}{2}t\right)^{h_n} (2t)^{-j_n} \\
&  \geq  \left( \frac{M_1 M_4^{\frac{1}{p-1}} t^{\frac{1-a}{p-1}}  }{ (2p)^{\frac{2p}{(p-1)^2}} 2^{\frac{1+a}{p-1}}} \right)^{p^{n-1}} (2p)^{\frac{2n(p-1)+2}{(p-1)^2}} M_4^{-\frac{1}{p-1}}
2^{\frac{1+a}{p-1}} t^{\frac{-1+a}{p-1}}.
\end{align*}
As above, if
$$ \frac{M_1 M_4^{\frac{1}{p-1}} t^{\frac{1-a}{p-1}}  }{ (2p)^{\frac{2p}{(p-1)^2}} 2^{\frac{1+a}{p-1}}}  >1,  $$
then $ G(t) \rightarrow \infty$ as $ n\rightarrow \infty $. Hence, we obtain that the upper bound of lifespan satisfies
$$ T \leq (M_1^{-1} M)^{\frac{p-1 }{1-a}},$$
with
$$ M= \frac{(2p)^{\frac{4p-2}{(p-1)^2}} 2^{\frac{(1+a)(p-1)+p^{n_{0}}}{(p-1)^2}}} {\left(   M_2(p-1) \right)^{\frac{1}{p-1}}}, $$
and hence we finish the proof for the subcritical case.

\par\noindent
\textbf{The critical case: $a=1$}.

In this case, we set
\begin{align}\label{kn}
& k_n:=\sum_{i=0}^n\frac{1}{2^i}.
\end{align}
Then for $E\ge 1$, we are going to prove
\begin{align}\label{cgt}
& G(t) \geq L_n \left(\log \frac{t}{k_n E} \right)^{q_n}, \ \text{for}\ t\geq k_n E
\end{align}
with the recurrence relations
\begin{equation}
\label{seqc}
\left\{
\begin{aligned}
& q_{n+1} = q_n p +1,  \ \  q_0=0, \\
& L_{n+1} = \frac{M_2 L_n^p}{4(2^n+1)(q_{n+1})}, \ \ L_0=M_1,
\end{aligned}
\right.
\end{equation}
where $ q_n \geq 0 $. Substituting \eqref{cgt} into \eqref{2lem12}, we get
\begin{align*}
G(t)& \geq  M_2 \int_{E}^t e^{-2s}ds\int^s_{E} \frac{e^{2r}|G(r)|^p} {1+r}dr\\
& \geq M_2 L_n^p \int_{k_n E}^t e^{-2s}ds\int^s_{k_n E} \frac{e^{2r}  \left(\log (t/(k_n E))\right)^{q_np}   } {1+r}dr\\
& \geq \frac{M_2 L_n^p}{2} \int_{k_n E}^t \frac{e^{-2s}}{s}ds \int^s_{k_n E} e^{2r}
\left(\log \frac{r}{k_n E}\right)^{q_np} dr
\end{align*}
for $ t\geq k_n E $. This further yields
\begin{equation}\label{Gt1}
	\begin{aligned}
G(t) & \geq \frac{M_2 L_n^p}{2} \int_{k_{n+1} E}^t \frac{e^{-2s}}{s}ds \int^s_{k_ns/k_{n+1}} e^{2r} \left(\log \frac{r}{k_n E}\right)^{q_np} dr\\
& \geq \frac{M_2 L_n^p}{2} \int_{k_{n+1}E}^t \frac{e^{-2s}}{s} \left(\log \frac{s}{k_{n+1} E}\right)^{q_np}ds \int^s_{k_ns/k_{n+1}} e^{2r} dr\\
& \geq \frac{M_2 L_n^p}{4} \int_{k_{n+1}E}^t \frac{1-e^{2(k_n/k_{n+1}-1)s}}{s} \left(\log \frac{s}{k_{n+1}E}\right)^{q_np}ds
\end{aligned}
	\end{equation}
for $ t\geq k_{n+1} E$. Similarly to \eqref{26-2}, we obtain
\begin{align*}
1-e^{2(k_n/k_{n+1}-1)s} &\ge1-e^{2(k_n-k_{n+1})E}\\
&\ge1-\frac{1}{1+2(k_{n+1}-k_n)E}\\
& =\frac{2(k_{n+1}-k_n)E}{1+2(k_{n+1}-k_n)E}\\
&\geq \frac{1}{2^n+1},
\end{align*}
which leads to by combining \eqref{Gt1}
\begin{align*}
G(t)& \geq  \frac{M_2 L_n^p}{4(2^n+1)}\int_{k_{n+1} E}^t \frac{1}{s} \left(\log \frac{s}{k_{n+1} E}\right)^{q_np}ds\\
& \geq \frac{M_2 L_n^p}{4(2^n+1) (q_np+1)}  \left(\log \frac{t}{k_{n+1} E}\right)^{q_np+1},
\end{align*}
which completes the proof for \eqref{cgt} with \eqref{seqc}.

It follows from $\eqref{seqc}_1$ that
\begin{align}\label{qn}
& q_n= p^{n} \left(  \frac{1}{p-1}\right)-\frac{1}{p-1},
\end{align}
which means
\begin{align*}
& q_{n+1} \leq  p^{n+1} \left(  \frac{1}{p-1}\right),
\end{align*}
and hence
\begin{align*}
\log L_{n+1} &\geq \log \frac{M_2}{4} + p \log L_n -\log(2^n+1)-\log  \left(p^{n+1}  \frac{1}{p-1}\right) \\
& \geq p \log L_n  - (n+1)\log(2p) - \log M_5,
\end{align*}
where
$$ M_5 := \frac{4}{M_2(p-1)}. $$
Similar to \cite{SST2025}, we obtain
\begin{align}\label{ln2}
\log L_n \geq p^n \log \left( \frac{M_1}{(2p)^{\frac{p}{(p-1)^2}} M_5^{\frac{1}{p-1}}} \right) + n \frac{\log (2p)}{p-1} + \log \left( (2p)^{\frac{p}{(p-1)^2}} M_5^{\frac{1}{p-1}} \right).
\end{align}
Combining \eqref{cgt} and \eqref{ln2} yield
\begin{align*}
G(t) &\geq  L_n \left(\log \frac{t}{k_n E} \right)^{q_n} \\
&\geq (2p)^{\frac{p+(p-1)}{(p-1)^2}} M_5^{\frac{1}{p-1}} \left( \frac{M_1}{(2p)^{\frac{p}{(p-1)^2}} M_5^{\frac{1}{p-1}}} \right)^{p^n}  \left( \log \frac{t}{k_n E} \right)^{q_n}.
\end{align*}
Setting
\[
 k_{\infty} =\sum_{i=0}^\infty\frac{1}{2^i},
 \]
hence we have for $ t\ge (k_\infty E)^2 $
\begin{align*}
G(t) &\geq (2p)^{\frac{p+(p-1)}{(p-1)^2}} M_5^{\frac{1}{p-1}} \left( \frac{M_1}{(2p)^{\frac{p}{(p-1)^2}} M_5^{\frac{1}{p-1}}} \right)^{p^n}  \left( \frac{1}{2} \log t \right)^{q_n}\\
& =  2^{\frac{1}{p-1}} (2p)^{\frac{p+(p-1)}{(p-1)^2}} M_5^{\frac{1}{p-1}}
(\log t)^{-\frac{1}{p-1}}  \left(\frac{M_1 (\log t )^{\frac{1}{p-1}} }{2^{\frac{1}{p-1}} (2p)^{\frac{p}{(p-1)^2}} M_5^{\frac{1}{p-1}}} \right)^{p^n}.
\end{align*}
If
$$  \frac{M_1 (\log t )^{\frac{1}{p-1}} }{2^{\frac{1}{p-1}} (2p)^{\frac{p}{(p-1)^2}} M_5^{\frac{1}{p-1}}} >1,  $$
then $ G(t) \rightarrow \infty$ as $ n\rightarrow \infty $, and hence the upper bound of lifespan satisfies
$$ T \leq\exp \left(  (M_1^{-1} M)^{p-1}  \right) $$
with
$$ M= 2^{ \frac{2p-1}{(p-1)^2}} p^{\frac{p}{(p-1)^2}} \left(\frac{4}{M_2(p-1)}\right)^{\frac{1}{p-1}}. $$
This completes the proof of Lemma \ref{2lem1} for the critical case.
\hfill$\square$

With the key lemma (Lemma \ref{2lem1}) in hand, we are going to prove Theorem \ref{thm2}. Following the idea in \cite{LLZ2017},
we introduce
\begin{equation}\label{phi}
\phi(x):=e^x+e^{-x},\quad \psi(x):=-e^x+e^{-x}.
\end{equation}
It is easy to see that
\[
\phi_x(x)=-\psi(x),\quad \phi(x)=-\psi_x(x).
\]
Multiplying the equation \eqref{IVP10} by $\phi(x)$ and integrating it in $\R$, we have
\[
\frac{d^2}{dt^2}\int_{\R}\phi(x)u(x,t)dx-\int_{\R}\psi(x)u_x(x,t)dx
=\int_{\R}\frac{|u_x(x,t)|^p}{\langle x\rangle^a}\phi(x)dx,
\]
due to the compactness of the support of $u$ and integration by parts.
Define
\[
F(t):=\int_{\R}\psi(x)u_x(x,t)dx,
\]
then the equation above can be rewritten as
\begin{equation}\label{F}
F''(t)-F(t)=\int_{\R}\frac{|u_x(x,t)|^p}{\langle x\rangle^a}\phi(x)dx.
\end{equation}

\par
On the other hand, by H\"{o}lder inequality it holds
\begin{equation}\label{Nonlin}
\begin{aligned}
|F(t)|&\le\left(\int_{\R}\frac{|u_x(x,t)|^p}{\langle x\rangle^a}\phi(x)dx\right)^{\frac1p}\left(\int_{\R}\phi(x)
\langle x\rangle^{\frac{a}{p-1}}dx\right)^{\frac{p-1}{p}}\\
&\leq 4^{\frac{p-1}{p}} \left(\int_{\R}\frac{|u_x(x,t)|^p}{\langle x\rangle^a}\phi(x)dx\right)^{\frac1p}\left(\int_0^{t+R}e^x(1+x)^{\frac{a}{p-1}}
dx\right)^{\frac{p-1}{p}}\\
&\leq (4C)^{\frac{p-1}{p}} (2R)^{\frac{a}{p}}
(1+t)^{\frac {a}p}e^{\frac{(p-1)t}{p}}\left(\int_{\R}\frac{|u_x(x,t)|^p}{\langle x\rangle^a}\phi(x)dx\right)^{\frac1p},
\end{aligned}
\end{equation}
which yields
\begin{align}\label{non1}
\int_{\R}\frac{|u_x(x,t)|^p}{\langle x\rangle^a}\phi(x)dx& \geq  M_6 \frac{|F(t)|^p}{(1+t)^{a}e^{(p-1)t}},
\end{align}
where
\[
M_6:= (4C)^{-(p-1)} (2R)^{-a}.
\]
Here we have employed Lemma 3.1 in Lai and Tu \cite{LT2020} to estimate
\[
\int_0^{1+t}e^x(1+x)^{\frac{a}{p-1}}dx
\le e^t\int_0^{1+t}e^{-(t-x)}(1+x)^{\frac{a}{p-1}}dx
\le Ce^t(1+t)^{\frac{a}{p-1}},
\]
where $C$ is the lemma.

\par
By combining \eqref{F} and \eqref{non1}, we come to
\begin{equation}\label{F1}
F''(t)-F(t)\geq M_6 \frac{|F(t)|^p}{(1+t)^{a}e^{(p-1)t}},
\quad\mbox{for}\ t\ge0.
\end{equation}
Setting
\[
G(t):=e^{-t}F(t),
\]
then we have
\[
F'(t)=e^tG(t)+e^tG'(t),\quad F''(t)=e^tG(t)+2e^tG'(t)+e^tG''(t),
\]
so that the inequality \eqref{F1} can be rewritten for $G(t)$ as
\begin{equation}\label{Gt}
G''(t)+2G'(t)\geq M_6 \frac{|G(t)|^p}{(1+t)^{a}},
\quad\mbox{for}\ t\ge0,
\end{equation}
which will give us the blow-up result and lifespan estimate in Theorem \ref{thm2}.
In fact, this inequality is exactly the same as (3.1) in Li and Zhou \cite{LZ1995},
and almost the same as the one in Lai and Takamura \cite{LT2018},
both papers study the semilinear damped wave equations.
Anyway, this inequality gives us
\[
e^{2t}G'(t)\geq M_6 \int_0^t\frac{e^{2s}|G(s)|^p}{(1+s)^{a}}ds+G'(0),
\]
which implies
\begin{align*}
G(t)& \geq M_6 \int_0^te^{-2s}ds\int_0^s\frac{e^{2r}|G(r)|^p}{(1+r)^{a}}dr
+\int_0^te^{-2s}dsG'(0)+G(0)\\
&  \geq M_6 \int_1^te^{-2s}ds\int_1^s\frac{e^{2r}|G(r)|^p}{(1+r)^{a}}dr + M_7\e
\end{align*}
for $ t\ge 1 $, where
\[
M_7:= \frac{1-e^{-2}}{2} \int_{\R}(-e^x+e^{-x})\{g'(x)-f'(x)\}dx+ \int_{\R}(-e^x+e^{-x})f'(x)dx.
\]
Hence, we arrive at
\begin{align}\label{Gt2}
&  G(t)\geq M_6 \int_1^te^{-2s}ds\int_1^s \frac{e^{2r}|G(r)|^p} {(1+r)^{a}}dr, \ \text{for}\ t\geq 1,
\end{align}
and
\begin{align}\label{Gt3}
&  G(t)\geq M_7 \varepsilon, \ \text{for} \ t\geq 1.
\end{align}
Taking $E=1$ in Lemma \ref{2lem1} and applying it to \eqref{Gt2} and \eqref{Gt3} for $t\ge 1$, we conclude that the upper bound of lifespan of problem \eqref{IVP10} satisfies
\[
T(\e) \leq \left\{
\begin{array}{ll}
 C_3^{\frac{p-1}{1-a}}\e^{-(p-1)/(1-a)}, & \mbox{if}\ a<1,\\
\exp \left(  C_{4}^{p-1} \varepsilon^{-(p-1)}   \right), & \mbox{if}\ a=1.
\end{array}
\right.
\]
with
\[
C_3= \left\{
\begin{array}{ll}
(2p)^{\frac{4p-2}{(p-1)^2}} 2^{\frac{(1-a)(p-1+p^{n_{0}})}{(p-1)^2}} (1-a)^{\frac{1}{p-1}}  (M_6(p-1))^{-\frac{1}{p-1}}M_7^{-1},  & \mbox{if}\ a\leq0,\\
(2p)^{\frac{4p-2}{(p-1)^2}} 2^{\frac{(1+a)(p-1)+p^{n_{0}}}{(p-1)^2}}\left(M_6(p-1) \right)^{-\frac{1}{p-1}}M_7^{-1}, & \mbox{if}\ 0<a<1,\\
\end{array}
\right.
\]
and
$$ C_{4}= M_7^{-1}2^{ \frac{2p-1}{(p-1)^2}} p^{\frac{p}{(p-1)^2}} \left(\frac{4}{M_6(p-1)}\right)^{\frac{1}{p-1}}. $$
Moreover, $ \varepsilon $ satisfies $\e\in(0,\e_1]$, where
\begin{align}\label{e2}
 \e_1 =  \min \left\{ \left( \frac{2 l_\infty R_2}{C_3^{(p-1)/(1-a)}}  \right)^{-\frac{1-a}{p-1}}, \, \left(  \frac{2\log k_\infty R_2}{C_4^{p-1}} \right)^{-\frac{1}{p-1}}\right\}
\end{align}
in this case,
then we will have
\[
C_3^{\frac{p-1}{1-a}}\e^{-(p-1)/(1-a)} > 2 l_\infty R_2\
\mbox{and}\  C_{4}^{p-1} \varepsilon^{-(p-1)}   > (k_\infty R_2)^2.
\]
This completes the proof for the Theorem \ref{thm2}.
\hfill$\square$


\section{Proof of Theorem \ref{thm3}}
Following Haruyama, Sasaki and Takamura \cite{HST}
which corrects a small error in the definition of the function space in the existence part of
Sasaki, Takamatsu and Takamura \cite{STT2023} for a case of $a=0$
(see Remark 3.2 in \cite{HST}),
the desired solution in Theorem \ref{thm3} is constructed in a $L^\infty$ space.
Its proof will be similar to the one for
\[
u_{tt}-u_{xx}=\frac{|u_t|^p}{\langle x\rangle^a}
\]
by Kitamura, Morisawa and Takamura \cite{KMT2023},
which has an oversight on the H\"{o}lder countinuity of the nonlinear term although.
The existence of a classical solution in \cite{KMT2023} is available for all $p>1$
with minor modifications in its proof.

\par
For our case, let us recall the expression of a solution $u$ of (\ref{IVP10}) in (\ref{integral}).
Then one can apply the time-derivative to (\ref{integral}) and (\ref{linear}) to obtain
\begin{equation}
\label{integral_t-derivative}
u_t(x,t)=\e u_t^0(x,t)+L_a'(|u_x|^p)(x,t)
\end{equation}
and
\[
u_t^0(x,t)=\frac{1}{2}\{f'(x+t)-f'(x-t)+g(x+t)+g(x-t)\},
\]
where $L_a'$ for a function $v=v(x,t)$ is defined by
\begin{equation}
\label{nonlinear_derivative}
\begin{array}{ll}
L_a'(v)(x,t):=
&\d\frac{1}{2}\int_0^t\frac{v(x+t-s,s)}{\{1+(x+t-s)^2\}^{a/2}}ds\\
&+\d\frac{1}{2}\int_0^t\frac{v(x-t+s,s)}{\{1+(x-t+s)^2\}^{a/2}}ds.
\end{array}
\end{equation}
Therefore, $u_t$ is expressed in terms of $u_x$.
On the other hand, applying the space-derivative to (\ref{integral}) and (\ref{linear}),
we have a $u_x$-closed integral equation,
\begin{equation}
\label{integral_x-derivative}
u_x(x,t)=\e u_x^0(x,t)+\overline{L_a'}(|u_x|^p)(x,t)
\end{equation}
and
\[
u_x^0(x,t)=\frac{1}{2}\{f'(x+t)+f'(x-t)+g(x+t)-g(x-t)\},
\]
where $\overline{L_a'}$ for a function $v=v(x,t)$ is defined by
\begin{equation}
\label{nonlinear_derivative_conjugate}
\begin{array}{ll}
\overline{L_a'}(v)(x,t):=
&\d\frac{1}{2}\int_0^t\frac{v(x+t-s,s)}{\{1+(x+t-s)^2\}^{a/2}}ds\\
&-\d\frac{1}{2}\int_0^t\frac{v(x-t+s,s)}{\{1+(x-t+s)^2\}^{a/2}}ds.
\end{array}
\end{equation}
Moreover, one more space-derivative to (\ref{integral_t-derivative}) yields that
\[
\begin{array}{ll}
u_{tx}(x,t)=&\d\e u_{tx}^0(x,t)\\
&\d+pL_a'(|u_x|^{p-2}u_xu_{xx})(x,t)
-aL_{a+2}'(|u_x|^px)(x,t)
\end{array}
\]
and
\[
u_{tx}^0(x,t):=\frac{1}{2}\{f''(x+t)-f''(x-t)+g'(x+t)+g'(x-t)\}.
\]
Similarly, we have that
\[
\begin{array}{ll}
u_{tt}(x,t)=
&\d\e u_{tt}^0(x,t)+\frac{|u_x(x,t)|^p}{(1+x^2)^{a/2}}\\
&\d+p\overline{L_a'}(|u_x|^{p-2}u_xu_{xx})(x,t)
-a\overline{L_{a+2}'}(|u_x|^px)(x,t)
\end{array}
\]
and
\[
u_{tt}^0(x,t)=\frac{1}{2}\{f''(x+t)+f''(x-t)+g'(x+t)-g'(x-t)\}.
\]
Therefore, $u_{tx}$ is expressed in terms of $u_{xx}$ and $u_x$, so is $u_{tt}$.
Also we have
\begin{equation}
\label{integral_xx-derivative}
u_{xx}(x,t)=\e u_{xx}^0(x,t)
+p\overline{L_a'}(|u_x|^{p-2}u_xu_{xx})(x,t)
-a\overline{L_{a+2}'}(|u_x|^px)(x,t)
\end{equation}
and
\[
u_{xx}^0(x,t)=u^0_{tt}(x,t).
\]

\par
First, we note the following fact.

\begin{prop}
\label{prop:system}
Assume that $(f,g)\in C_0^2(\R)\times C_0^1(\R)$.
Let $w$ be a $C^1$ solution of (\ref{integral_x-derivative}) in which $u_x$ is replaced with $w$.
Then,
\[
u(x,t):= \int^x_{-\infty}w(y,t)dy
\]
is a classical solution of (\ref{IVP10}) in $\R\times[0,T]$.
\end{prop}
\par\noindent
{\bf Proof.} This is easy along with the computations above in this section.
\hfill$\Box$

\par
According to Proposition \ref{prop:system}, we shall construct a $C^1$ solution
of (\ref{integral_x-derivative}) in which $u_x$ is replaced with the unknown function $w$.
Let $\{w_j\}_{j\in \N}$ be a sequence of $C^1(\R\times[0,T])$ defined by
\begin{equation}
\label{w_j}
\left\{
\begin{array}{l}
w_{j+1}=\e u_x^0+\overline{L_a'}(|w_j|^p), \\
w_1=\e u_x^0.
\end{array}
\right.
\end{equation}
Then, in view of (\ref{integral_xx-derivative}), $(w_j)_x$ has to satisfy
\begin{equation}
\label{w_j_x}
\left\{
\begin{array}{l}
(w_{j+1})_x=\e u_{xx}^0+p\overline{L_a'}(|w_j|^{p-2}w_j(w_j)_x)(x,t)
-a\overline{L_{a+2}'}(|w_j|^px)(x,t), \\
(w_1)_x=\e u_{xx}^0,
\end{array}
\right.
\end{equation}
so that the functional space in which $\{w_j\}$ converges is
\[
X:=\{w\in C^1(\R\times[0,T]): \|w\|_X<\infty, \mbox{ supp }w \subset \{(x,t)\in\R\times[0,T] : |x|\leq t+R\}\}
\]
which is equipped with the norm
\[
\|w\|_X:=\|w\|+\|w_x\|
\]
where
\[
\|w\|:= \sup_{(x,t)\in\R\times[0,T]}|w(x,t)|.
\]
We note that (\ref{integral_x-derivative}) implies that
\[
\begin{array}{l}
\mbox{supp}\ w_j\subset\{(x,t)\in\R\times[0,T] : |x|\leq t+R\}\\
\Longrightarrow\mbox{supp}\ w_{j+1}\subset\{(x,t)\in\R\times[0,T] : |x|\leq t+R\}.
\end{array}
\]
\par
The following lemma provides us with an a priori estimate.
\begin{prop}
\label{prop:apriori}
Let $w\in C(\R\times[0,T])$ and supp\ $w\subset\{(x,t)\in\R\times[0,T]:|x|\leq t+R\}$.
Then, the following a priori estimate holds;
\begin{equation}
\label{apriori}
\|\overline{L'_a}(|w|^p)\|\leq C\|w\|^pE(T),
\end{equation}
where $C$ is a positive constant independent of $T$ and $\e$.
$E(T)$ is defined by
\begin{equation}
\label{E(T)}
E(T):=
\left\{
\begin{array}{cl}
(T+2R)^{1-a} & \mbox{if}\ a<1,\\
\log(T+2R) & \mbox{if}\ a=1,\\
1 & \mbox{if}\ a>1.
\end{array}
\right.
\end{equation}
\end{prop}

\par\noindent
{\bf Proof.} The proof of Proposition \ref{prop:apriori} is almost the same as
the one of Proposition 3.1 in Morisawa, Kitamura and Takamura \cite{KMT2023}
because the difference between $L'_a$ and $\overline{L'_a}$ is a sign of the second terms
in (\ref{nonlinear_derivative}) and (\ref{nonlinear_derivative_conjugate}),
so that we have
\[
\left|\overline{L'_a}(v)(x,t)\right|\le L'_a(|v|)(x,t)\quad\mbox{in}\ \R\times[0,T].
\]
Therefore the remained part is completely same as the one
of Proposition 3.1 in \cite{KMT2023}.
\hfill$\Box$

\par
Once Propositions above are established,
the remained part of the proof of Theorem \ref{thm3} is a routine work.
We shall omit the details by citing \cite{KMT2023} for basic computations,
Kido, Sasaki, Takamatsu and Takamura \cite{KSTT2024} for handling
the H\"{o}lder countinuity of the nonlinear term.
\hfill$\Box$

\section*{Acknowledgement}
This work was started when the second author visited Tohoku University during 8/7/2025-4/10/2025.
She appreciates the financial support from the International Office and School of Mathematical Sciences of Zhejiang Normal University.
Moreover, all the authors are partially supported by JSPS Bilateral Program Number JPJSBP120257401.
Also the first author is partially supported by NSFC (No.12271487, W2521007).
The third and fourth authors are partially supported by the Grant-in-Aid for Scientific Research(A) (No.22H00097)
and (C) (No.24K06819), Japan Society for the Promotion of Science.
Finally, the fourth author is partially supported by
the Science Committee of the Ministry of Education and Science of the
Republic of Kazakhstan (grant no.AP26195417).

All the authors appreciate the reviewer's thoughtful comments, which were very helpful in improving the quality and clarity of the manuscript.

\bibliographystyle{plain}

\end{CJK*}

\end{document}